\documentclass[11pt]{amsart}
\usepackage{color}
\def\rojo{}
\def\colorado{}
\def\colorados{}
\def\rouge{}
\def\red{}
\def\Rojo{}
\def\magenta{}

\makeatletter
\@namedef{subjclassname@2020}{%
  \textup{2020} Mathematics Subject Classification}
\makeatother

\usepackage{enumerate}
\usepackage{tikz}
\usepackage{tikz-cd}
\usetikzlibrary{arrows.meta,graphs}
\usepackage{oubraces}
\usepackage[all]{xy}
\usepackage{amsmath,amssymb}
\usepackage{graphicx}
\usepackage{sidenotes}
\usepackage[mathscr]{euscript}
\usepackage{chngcntr}
\usepackage[colorlinks,allcolors=blue]{hyperref}
\counterwithin{figure}{section}

\newtheorem{theorem}{Theorem}[section]
\newtheorem{proposition}[theorem]{Proposition}
\newtheorem{lemma}[theorem]{Lemma}
\newtheorem{corollary}[theorem]{Corollary}

\newtheorem{definition}[theorem]{Definition}

\theoremstyle{remark}
\newtheorem{example}[theorem]{Example}
\newtheorem{remark}[theorem]{Remark}

\newcommand{\thmref}[1]{Theorem~\ref{#1}}
\newcommand{\propref}[1]{Proposition~\ref{#1}}

\newcommand{\defref}[1]{Definition~\ref{#1}}

\def\cat{{\sf{cat}}}

\def\TC{{\mathsf{TC}}}

\def\LS{\protect\operatorname{LS}}

\def\CW{\protect\operatorname{CW}}
\def\relcat{\protect\operatorname{relcat}}
\def\ANR{\protect\operatorname{ANR}}
\def\ENR{\protect\operatorname{ENR}}
\def\toi{\hookrightarrow}

\newcommand{\und}{\underline}
\newcommand{\normados}{\mathcal{C}(G_1,\ldots,G_m; K)}
\newcommand{\polyg}{\underline{G}^K}

\allowdisplaybreaks
	
\begin{document}

\title[\colorado{Fadell-Husseini motion planning}]{\colorado{Motion planning in} polyhedral products of groups \colorado{and a Fadell-Husseini approach to topological complexity}}

\author{Jorge Aguilar-Guzm\'an}
\address{Departamento de Matem\'aticas\\
Centro de Investigaci\'on y de Estudios Avanzados del I.P.N.\\
M\'exico City 07000\\M\'exico}
\email{jaguzman@math.cinvestav.mx}

\author{Jes\'us Gonz\'alez}
\address{Departamento de Matem\'aticas\\
Centro de Investigaci\'on y de Estudios Avanzados del I.P.N.\\
M\'exico City 07000\\M\'exico}
\email{jesus@math.cinvestav.mx}

\begin{abstract}
We compute the topological complexity of a polyhedral product \colorado{$\mathcal{Z}$} defined \rouge{in terms of} an $\LS$-logarithmic family of \colorados{locally compact} connected $\CW$ topological groups. The answer is given \red{by a combinatorial formula that involves} the $\LS$ category of the polyhedral-product factors. As a by-product, we show that the Iwase-Sakai conjecture holds true for \colorado{$\mathcal{Z}$.} The proof methodology \rouge{uses} a Fadell-Husseini viewpoint for the monoidal topological complexity \colorado{(MTC)} of a space, which, under mild conditions, recovers Iwase-Sakai's original definition. \colorado{In the Fadell-Husseini context, the stasis condition \colorado{---MTC's \emph{raison d'\^etre}---} can be encoded at the covering level.} Our Fadell-Husseini-inspired definition provides an alternative to the \colorado{MTC} variant given by Dranishnikov, as well as to the ones provided by Garc\'ia-Calcines, Carrasquel-Vera and Vandembroucq in terms of relative category.
\end{abstract}

\keywords{Monoidal topological complexity, Iwase-Sakai conjecture, \colorado{Fadell-Husseini topological complexity,} polyhedral product, relative category.}

\subjclass[2020]{55M30, 55M15, 68T40}

\maketitle

\section{Introduction}

As originally introduced in~\cite{Iwase}, the monoidal topological complexity of a path-connected space~$X$, denoted \rouge{by} $\TC^M(X)$, is \colorado{apparently} more restrictive \colorado{than Farber's} topological complexity of $X$, $\TC(X)$. \rouge{For} a locally finite simplicial complex \rouge{$X$,} Iwase and Sakai showed in~\cite[Theorem~1]{Iwase2} that $\TC^M(X)$ differs from $\TC(X)$ by at most one unit, \colorado{and conjectured that the equality} $\TC^M(X)=\TC(X)$ \colorado{holds true.}

In~\cite[Lemma 2.7]{Dran} Dranishnikov proved Iwase-Sakai's conjecture \colorado{for} a connected Lie group \rouge{$G$,}
\begin{equation}\label{tctcmcat}
\TC(G)=\cat(G)=\TC^M(G),
\end{equation}
\colorado{where the equality between $\TC(G)$ and $\cat(G)$, the $\LS$ category of $G$, is well known.} The first main result in this paper generalizes~(\ref{tctcmcat}) to the realm of polyhedral products defined by $\LS$-logarithmic families of \colorados{locally compact} connected CW topological groups. In particular, we show that Iwase-Sakai's conjecture remains valid in such a context. Explicitly:

\begin{theorem}\label{TC(G^K)=TCM(G^K)}
Let $\polyg$ denote the polyhedral product associated to an abstract simplicial complex $K$ \magenta{with vertex set $\{1,2,\ldots, m\}$} and a based family $\underline{G}=\{(G_i,e_i)\}_{i=1}^m$ of \colorados{locally compact} connected $\CW$ topological groups, where $e_i$ stands for the neutral element of~$G_i$. If \rouge{$\underline{G}$ is an LS-logarithmic family, i.e., if the equality}
$$
\cat(G_{i_1}\times\cdots\times G_{i_k})=\cat(G_{i_1})+\cdots+\cat(G_{i_k})
$$
holds \rouge{true} for any strictly increasing sequence \red{$1\leq i_1<\cdots<i_k\leq m$,} then
\begin{equation}\label{tctcmasasum}
\TC(\polyg)=\TC^M(\polyg)=\max\bigg\{\sum_{i\in\sigma_{1}\cup\sigma_{2}}\cat(G_{i})\colon \sigma_1,\sigma_2\in K\bigg\}.
\end{equation}
\end{theorem}

\colorado{The expression in (\ref{tctcmasasum}) should be compared to the fact that the} $\LS$ category of $\polyg$ is \colorado{also determined by} the $\LS$ category of the polyhedral-product factors. Namely, \red{under} the hypotheses of \thmref{TC(G^K)=TCM(G^K)}, \red{\cite[Theorem~1.4]{AGO} gives} 
\begin{equation}\label{cat of G^K}
\cat(\polyg)=\max\bigg\{\sum_{i\in\sigma}\cat(G_{i})\colon \sigma\in K\bigg\}.
\end{equation}

The proof of \thmref{TC(G^K)=TCM(G^K)} involves a Fadell-Husseini \colorado{flavored} definition of monoidal topological complexity of a space. \colorado{Such a concept} arises naturally \red{by} inputting Dranishnikov's and Garc\'ia-Calcines' views \colorado{---reviewed below---} into Iwase-Sakai's original definition.

Dranishnikov noticed in~\cite{Dran} that, if $X$ is an $\ENR$ space (or more generally an $\ANR$ space as we will see in \propref{equivalence of definitions} below), \red{Iwase-Sakai's} definition of $\TC^M(X)$ can be relaxed in the sense that the diagonal of $X$ does not have to be contained in each open domain covering $X\times X$ \rouge{(yet,} the continuous local sections of the end-points evaluation map $\colorado{e_{01}}\colon X^{[0,1]}\to X\times X$ are still required to yield constant paths on points of the diagonal). Furthermore, in~\cite[Remark~2.19]{JoseManuel}, Garc\'ia-Calcines \colorado{proved} that, when $X$ is an $\ANR$ space, $\TC^M(X)$ can be defined in terms of general \red{(not necessarily open) covers} of $X\times X$ \red{by} following the \red{lines} in Iwase-Sakai's work, that is, \red{by} requiring that the diagonal lies in each subset covering $X\times X$, and that the corresponding local sections yield constant paths when restricted to the diagonal of $X$. On the other hand, Fadell and Husseini defined in~\cite{Fadell} the relative category of a cofibered pair $(X,A)$, denoted \rouge{by} $\cat^{FH}(X,A)$, in terms of open sets covering $X$. In contrast \colorado{to} the definition of $\cat(X)$, \rouge{in the relative-category  context} \colorado{only} one of \colorado{the covering} open subsets \colorado{is required} to contain $A$ and deform to $A$ (rel $A$) within $X$, while the rest of the open sets are actually required to deform within $X$ to a point.

We introduce the Fadell-Husseini monoidal topological complexity of a space $X$ by \colorado{blending} the definition of $\cat^{FH}(X,A)$ \colorado{into} Dranishnikov's and Garc\'ia-Calcines' \colorado{viewpoints for} $\TC^M(X)$. Explicitly:

\begin{definition}\label{def:FHGMTC} \ 

\begin{itemize}
\item[(a)]The Fadell-Husseini monoidal topological complexity of a path connected space $X$, denoted \rouge{by} $\TC^{\colorado{FH}}(X)$, is the smallest \red{nonnegative integer} $n$ for which there is an open \red{cover} $\{U_0,\ldots,U_n\}$ of $X\times X$ by $n+1$ subsets, on each of which there exists a continuous section $s_i\colon U_i\rightarrow X^{[0,1]}$ of the end-points evaluation map \colorado{$e_{01}$ so that:}
\begin{itemize}
\item[(1)] \colorado{$U_{0}$ contains the diagonal $\Delta X=\{(x,x)\colon x\in X\};$}
\item[(2)] $s_{0}(x,x)=c_x$, the constant path at $x$, for all $x\in X;$
\item[(3)] $\Delta X\cap U_{i}=\varnothing$ for all $i\geq 1$.
\end{itemize}
\item[(b)]\colorado{The Fadell-Husseini \emph{generalized} monoidal topological complexity $\TC^{\colorado{FH}}_g(X)$ is defined as above, except that the elements of the covers $\{U_0,\ldots,U_n\}$ are not required to be open.}
\end{itemize}
\colorado{We agree to set \colorado{$\TC^{\colorado{FH}}(X)=\infty$ or $\TC^{\colorado{FH}}_g(X)=\infty$, if the required} coverings fail to exist.}
\end{definition}

The second main result in this paper asserts that \defref{def:FHGMTC} recovers the ones given by Iwase-Sakai ($\TC^M$), Dranishnikov ($\TC^{DM}$) and Garc\'ia-Calcines ($\TC^M_g$) when working with $\ANR$ spaces. 
\begin{theorem}\label{Iwase=Dranish=GC}
If $X$ is an $\ANR$ space, then 
$$
\TC^{\colorado{FH}}(X)=\TC^{\colorado{FH}}_{g}(X)=\TC^{DM}(X)=\TC^{M}(X)=\TC^M_g(X).
$$
\end{theorem}

\colorado{As noted above, the last equality in Theorem~\ref{Iwase=Dranish=GC} is due to Garc\'ia-Calcines, whereas the next-to-last equality is due to Dranishnikov (see Proposition~\ref{equivalence of definitions} below).}

\colorado{In the generalized setting, condition~(3) in \defref{def:FHGMTC} can be omitted without altering the value of $\TC^{FH}_g(X)$, as the diagonal \red{$\Delta X$} can be removed, if needed, from the sets $U_1,\ldots,U_n$. A similar observation applies in the non-generalized setting as long as $X$ is a Hausdorff space. Far more striking is the role of the stasis condition (2) in Definition~\ref{def:FHGMTC}. The next result \rouge{was pointed out to the authors by} J.~M.~Garc\'ia-Calcines:}

\begin{theorem}\label{dejosemanuel}
\colorados{If $X$ is a locally equiconnected  Hausdorff space such that $X\times X$ is normal, \rouge{then} the stasis condition $(2)$ \rouge{in Definition~\ref{def:FHGMTC}} can be \rouge{ignored} without altering the \rouge{resulting} numerical value of $\TC^{FH}(X)$. The same conclusion holds in the generalized setting \rouge{if} $X$ is an $\ANR$ space.}
\end{theorem}

\colorado{Consequently, the equality $\TC(X)=\TC^M(X)$ \rouge{in the Iwase-Sakai conjecture holds true} for any ANR space $X$ for which there is a (not necessarily monoidal) motion planner with $\TC(X)+1$ (not necessarily open) local domains one of which contains the diagonal $\Delta X$ (cf.~Corollary~3 in \cite{Iwase2}).}

The proof of \thmref{Iwase=Dranish=GC} \red{follows} the guideline established by Garc\'ia-Calcines in~\cite{JoseManuel} to show that $\TC^M(X)=\TC_g^M(X)$ when dealing with an $\ANR$ space $X$. In fact, paralleling many of his techniques, we introduce in Definitions~\ref{def:relcatFH} and~\ref{def:relcatgFH} below the notions of $\relcat^{FH}_{op}(i_X)$ and $\relcat^{FH}_{g}(i_X)$, where $i_X:A\toi X$ stands for a cofibration. Both concepts represent a Fadell-Husseini version of the concepts $\relcat_{op}(i_X)$ and $\relcat_{g}(i_X)$, which were widely studied in~\cite{JoseManuel} in order to provide a characterization by covers of relative category in the sense of Doeraene-El Haouari. We show that if $X$ is an $\ANR$ space and $i_X:\Delta X\toi X\times X$ denotes the canonical cofibration, then
\begin{equation}\label{all relcat are equal}
\relcat^{FH}_g\!(i_X\hspace{-.3mm})=\relcat_{op}^{FH}\!(i_X\hspace{-.3mm})=\relcat(i_X\hspace{-.3mm})=\relcat_{op}(i_X\hspace{-.3mm})=\relcat_g(i_X\hspace{-.3mm}),
\end{equation}
where the last two equalities were \red{proved} in~\cite[Theorems~1.6 and 2.16]{JoseManuel}. Finally, the crucial ingredient to prove \thmref{Iwase=Dranish=GC} comes from~\cite[Theorem~12]{CGV}, where the central term in~(\ref{all relcat are equal}) is shown to agree with $\TC^M(X)$.

The authors gratefully acknowledge Jos\'e Manuel Garc\'ia-Calcines for a careful reading and valuable comments on a preliminary version of this work, and for his generous permission to include here his Theorem~\ref{dejosemanuel}.

\section{Preliminaries}
\subsection{Topological complexity and its monoidal version}

\begin{definition}[Farber]\label{definition of tc}
\colorado{For a path-connected space $X$, let $e_{01}\colon X^{[0,1]}\rightarrow X\times X$ be the end-points evaluation map that sends each path $\gamma$ in $X$ into the ordered pair $(\gamma(0),\gamma(1))$.} The topological complexity of $X$, denoted \rouge{by} $\TC(X)$, is defined as the smallest $n$ for which there is an open cover $\{U_0,\ldots,U_n\}$ of $X\times X$ by $n+1$ open sets, on each of which there is a continuous local section $s_i\colon U_i\rightarrow X^{[0,1]}$ of \colorado{$e_{01}$.} If the coverings fail to exist, we agree to set $\TC(X)=\infty$. The \rouge{open} sets $U_i$ covering $X\times X$ and the corresponding local sections $s_i$ of \colorado{$e_{01}$} are called local domains and local rules, respectively, whereas the family $\{(U_i,s_i)\}$ is known as a motion planner.
\end{definition}

As we pointed out at the beginning of the introduction, monoidal topological complexity is \colorado{an apparently} stronger version of $\TC$ satisfying an additional condition: if $X$ is thought of as the configuration space of a mechanical system and $\red{A,B}\in X$ \red{is a pair of initial-final configurations of the system with $A=B$, then} the continuous motion of the system from $A$ to $B$ is required to be the constant path at $A$. Such a requirement appears to be quite natural in actual applications. Explicitly:
\begin{definition}[Iwase-Sakai]\label{MTC}
The monoidal topological complexity of a path-connected space $X$, denoted \rouge{by} $\TC^{M}(X)$, is the smallest $n$ for which there is an open cover $\{U_0,\ldots,U_n\}$ of $X\times X$ by $n+1$ open sets, each one containing the diagonal $\Delta X=\{(x,x)\colon x\in X\}$, and on each of which there is a continuous local section $s_i\colon U_i\rightarrow X^{[0,1]}$ of the end-points evaluation map $\colorado{e_{01}}\colon X^{[0,1]}\rightarrow X\times X$ such that, for each $x\in X$, $s_i(x,x)=c_{x}$, the constant path at $x$. Such a section is called reserved. If the coverings fail to exist, we agree to set $\TC^{M}(X)=\infty$.
\end{definition}

The example exhibited in~\cite[p.13]{JoseManuel} shows that, \rouge{unlike the usual topological complexity,}  the monoidal topological complexity \rouge{fails to be} a homotopy invariant. \red{It is known that} $\TC^{M}(X)$ is a homotopy invariant if $X$ is locally equiconnected, i.e., provided the canonical \rouge{diagonal} embedding $\Delta X\hookrightarrow X\times X$ is a cofibration (see~\cite[Theorem 12]{CGV} and~\cite[Proposition 2.17]{JoseManuel}). An important instance of \colorado{locally} equiconnected spaces is given by an absolute neighborhood retract ($\ANR$) which, throughout this paper, means a metrizable space $X$ satisfying the following property: every \rouge{continuous} map $f\colon A\rightarrow X$ \rouge{defined on a closed subset $A$ of a metrizable space~$Y$} can be \rouge{continuously} extended over an open neighborhood $U$ of $A$ in $Y$.

\colorado{The} next proposition claims that the condition $\Delta X\subseteq U_{i}$, imposed \rouge{on} each set of the open cover $\{U_{i}\}_{i=0}^n$ of $X\times X$ in Definition~\ref{MTC}, can be omitted in the case when $X$ is an $\ANR$ space.

\begin{proposition}\label{equivalence of definitions}
If $X$ is an $\ANR$ space, then $\TC^M(X)=\TC^{DM}(X)$, where the latter expression is defined to be the smallest nonnegative integer $n$ for which there is an open cover $\{U_0,\ldots,U_n\}$ of $X\times X$, on each of which there is a continuous local section $s_i\colon U_i\rightarrow X^{[0,1]}$ of the fibration $\colorado{e_{01}}\colon X^{[0,1]}\rightarrow X\times X$ such that $s_i(x,x)=c_{x}$ for all $x\in X$ with $(x,x)\in U_i$.
\end{proposition}
\begin{proof}
Clearly $\TC^{DM}(X)\leq \TC^M(X)$. \red{We} show the opposite inequality \colorado{by following the indications at the bottom of page 4365 of~\cite{Dran}.}

Let $\{U_0,\ldots,U_n\}$ be an open cover of $X\times X$ by sets that admit continuous local sections $s_i\colon U_i\rightarrow X^{[0,1]}$ of the fibration \colorado{$e_{01}$} such that $s_i(x,x)=c_{x}$ for all $x\in X$ with $(x,x)\in U_i$. Since $X\times X$ is a normal space, we can assure the existence of a closed cover $\{V_{0},\ldots,V_{n}\}$ of $X\times X$ with $V_{i}\subseteq U_{i}$ for all $i\in\{0,1,\ldots,n\}$. Then we have a continuous extension $$\overline{s}_{i}\colon V_{i}\cup\Delta X\rightarrow X^{[0,1]}$$ of $s_{i}|_{V_i}$ defined by
\[
\overline{s}_{i}(x,x')=\begin{cases} s_{i}(x,x'), & \text{ if } (x,x')\in V_{i}; \\
c_{x},&\text{ if } (x,x')\in \Delta X.
\end{cases}
\]
For $i\in\{0,1,\ldots,n\}$, let $\Gamma_{i}$ stand for the closed subset $((V_{i}\cup\Delta X)\times [0,1])\cup (X\times X\times\{0,1\})$ of $X\times X\times [0,1]$ and define \red{a} continuous function $u_{i}\colon\Gamma_{i}\rightarrow X$ \red{by}
\[
u_{i}(x,x',t)=\begin{cases}
\overline{s}_{i}(x,x')(t),&\text{ if } (x,x',t)\in (V_{i}\cup\Delta X)\times [0,1];\\
x,&\text{ if } (x,x',t)\in X\times X\times \{0\};\\
x',&\text{ if } (x,x',t)\in X\times X\times \{1\}.
\end{cases}
\]
Note that $u_{i}$ is well-defined because $\overline{s}_{i}$ is a continuous local section of \colorado{$e_{01}$.} Since $X$ is an $\ANR$ space, there are open neighborhoods $W_{i}$ of $\Gamma_i$ in $X\times X\times [0,1]$ and continuous maps $\overline{u}_{i}\colon W_{i}\rightarrow X$ with $\overline{u}_{i}|_{\Gamma_i}=u_{i}$. By the compactness of $[0,1]$, \red{we} can take an open set $N_{i}$ in $X\times X$ containing $V_{i}\cup\Delta X$ such that $N_{i}\times [0,1]\subseteq W_{i}$. Finally, the required reserved section $s'_{i}\colon N_{i}\rightarrow X^{[0,1]}$ of the fibration \colorado{$e_{01}$} is defined by $s'_{i}(x,x')(t)=\overline{u}_{i}(x,x',t)$ for all $(x,x')\in N_{i}$ and $t\in[0,1]$. Indeed, by construction, $s'_{i}$ is a continuous extension of $\overline{s}_{i}$. Therefore, the new open cover $\{N_{0},\ldots,N_{n}\}$ of $X\times X$ fulfills the requirements of Definition~\ref{MTC}.
\end{proof}
\begin{remark}\label{TCM=TCDM=TCFHM}
For a Hausdorff space $X$ (so that $\Delta X$ is closed), the inequalities $\TC^{DM}(X)\leq\TC^{\colorado{FH}}(X)\leq\TC^M(X)$ follow directly from the definitions. Furthermore, these inequalities are in fact equalities if $X$ is an $\ANR$, in view of \propref{equivalence of definitions}.
\end{remark}

The proof of the equalities $\TC(G)=\cat(G)=\TC^M(G)$ given in~\cite[Lemma 2.7]{Dran}, where $G$ is a \Rojo{connected} Lie \red{group}, can be adapted to show that the Iwase-\red{Sakai} conjecture holds \magenta{true} for a \colorados{locally compact} connected $\CW$ topological group. We spell out the details (in Proposition~\ref{MTC of a CW-TG} below) as they will be useful in \colorado{Section}~\ref{applications} for constructing an explicit motion planner leading to the inequality 
$$
\TC^M(\polyg)\leq\max\bigg\{\sum_{i\in\sigma_{1}\cup\sigma_{2}}\cat(G_{i})\colon \sigma_1,\sigma_2\in K\bigg\}
$$
in \thmref{TC(G^K)=TCM(G^K)}.

\begin{proposition}\label{MTC of a CW-TG}
If $G$ \rouge{is} a \colorados{locally compact} connected $\CW$ topological group, then $\TC(G)=\cat(G)=\TC^{M}(G)$. 
\end{proposition}
\begin{proof}
Since $\TC(G)\leq\TC^M(G)$, \red{with} the former \red{agreeing} with $\cat(G)$ (see \cite[Lemma 8.2]{Farber2}, where the same proof works for topological groups), it suffices to show that $\TC^{M}(G)\leq \cat(G)$. Furthermore, since a \colorados{locally compact $\CW$ complex is an $\ANR$ space} (see Appendix II of \cite{Lundell}), we \colorado{only need to} show that $\TC^{\colorado{FH}}(G)\leq\cat(G)$, \colorado{in view of} Remark~\ref{TCM=TCDM=TCFHM}.

\rouge{Let $n:=\cat(G)$ and choose} an open cover $\{N_0,\ldots,N_n\}$ of $G$ together with homotopies $H_{i}\colon N_{i}\times[0,1]\rightarrow G$ satisfying $H_{i}(a,0)=a$ and $H_{i}(a,1)=e$, $a\in N_i$, for all $i\in\{0,1\Rojo{,}\ldots,n\}$ (here $e$ denotes the neutral element of $G$). We can assume that $e\notin N_i$ for all $i>0$ and $H_{0}(e,t)=e$ for all $t\in [0,1]$, where the latter requirement follows from~\cite[Lemma 1.25]{CLOT} and the fact that $\{e\}\hookrightarrow G$ is a cofibration (recall \rouge{that} $\CW$ complexes have non-degenerate base points). For each $i\in\{0,1\Rojo{,}\ldots,n\}$, set $V_{i}:=\{(a,b)\in G\times G\colon b^{-1}a\in N_i\}$. On each $V_i$ of the open cover $\{V_{0},\ldots,V_{n}\}$ of $G\times G$ \magenta{we have the} continuous section $s_{i}\colon V_{i}\rightarrow G^{[0,1]}$ of \colorado{$e_{01}$} defined by $s_{i}(a,b)(t)=bH_{i}(b^{-1}a,t)$, $t\in[0,1]$. Note that $\Delta G\cap V_i=\varnothing$ for all $i\in \{1,\ldots,n\}$ and $s_0(a,a)(t)=aH_{0}(a^{-1}a,t)=aH_{0}(e,t)=ae=a$ with $(a,a)\in\Delta G$ and $t\in[0,1]$. Therefore $\TC^{\colorado{FH}}(G)\leq \cat(G)$ and the result follows.
\end{proof}

\subsection{Relative category}
We start by recalling the definition of the join of two maps having the same target. Let $f\colon X\to Z$ and $g\colon Y\to Z$ be maps, the join of $f$ and $g$, denoted \rouge{by} $X\ast_Z Y$, is defined as the homotopy pushout of the homotopy pullback of $f$ and $g$,
\begin{center}
\begin{tikzpicture}[commutative diagrams/every diagram]
\matrix[matrix of math nodes, name=m, commutative diagrams/every cell, column sep=4ex]
    {\bullet&&Y\\
     &X\ast_Z Y&\\
    X&&Z,\\};
\path[commutative diagrams/.cd, every arrow, every label]
(m-1-1) edge (m-1-3)
(m-1-1) edge (m-3-1)
(m-3-1) edge (m-2-2)
(m-1-3) edge node[right] {$g$} (m-3-3)
(m-1-3) edge (m-2-2)
(m-3-1) edge node[below] {$f$} (m-3-3)
(m-2-2) edge[commutative diagrams/dashed] (m-3-3);
\end{tikzpicture}
\end{center}
where the dashed arrow, called join map or whisker map, is given by the weak universal property of the homotopy pushout.

The previous definition enables us to set forth the main notion of this section, the relative category of a map as introduced in~\cite{DoeHao}. \colorado{The $n$-th Ganea map of a given map $i_X\colon A\to X$, denoted \rouge{by}} $g_n\colon G_n(\colorado{i_X})\to X$ $(n\geq 0)$, is the join map inductively defined by the join construction
\begin{center}
\begin{tikzpicture}[commutative diagrams/every diagram]
\matrix[matrix of math nodes, name=m, commutative diagrams/every cell, column sep=4ex]
    {\bullet&&A\\
     &G_n(\colorado{i_X})&\\
    G_{n-1}(\colorado{i_X})&&X,\\};
\path[commutative diagrams/.cd, every arrow, every label]
(m-1-1) edge (m-1-3)
(m-1-1) edge (m-3-1)
(m-3-1) edge (m-2-2)
(m-1-3) edge node[right] {$i_X$} (m-3-3)
(m-1-3) edge node[above left] {$\alpha_n$} (m-2-2)
(m-3-1) edge node[below] {$g_{n-1}$} (m-3-3)
(m-2-2) edge[commutative diagrams/dashed] node[above right] {$g_{n}$} (m-3-3);
\end{tikzpicture}
\end{center}
where $g_0:=i_X$ and, if $g_{n-1}\colon G_{n-1}(\colorado{i_X})\to X$ is already given, $G_n(\colorado{i_X})$ is the join of $i_X$ and $g_{n-1}$. Then, the \textit{relative category} of $i_X$, denoted \rouge{by} $\relcat(i_X)$, is defined as the least nonnegative integer~$n$ such that $g_n\colon G_n(\colorado{i_X})\to X$ admits a homotopy section $\sigma\colon X\to G_n(\colorado{i_X})$ satisfying $\sigma\circ i_X\simeq \alpha_n$. 

Doeraene and El Haouari proved in~\cite{DoeHao} that the relative category of a map possesses a \rouge{Whitehead-type} characterization in terms of the $n$-th sectional fat-wedge $t_n\colon T^n(i_X)\to X^{n+1}$ of $i_X$ (see Theorem~\ref{thm:Doe-Hao} below), which is inductively defined as follows: For $n=0$, set $T^0(i_X):=A$ and $t_0=i_X\colon A\to X$. If $t_{n-1}\colon T^{n-1}(i_X)\to X^n$ is already defined, then $t_n$ is the join map \rouge{rendering a homotopy} commutative diagram
\begin{center}
\begin{tikzpicture}[commutative diagrams/every diagram]
\matrix[matrix of math nodes, name=m, commutative diagrams/every cell, column sep=4ex]
    {\bullet&&X^n\times A\\
     &T^n(i_X)&\\
    T^{n-1}(i_X)\times X&&X^{n+1},\\};
\path[commutative diagrams/.cd, every arrow, every label]
(m-1-1) edge (m-1-3)
(m-1-1) edge (m-3-1)
(m-3-1) edge (m-2-2)
(m-1-3) edge node[right] {$1_{X^n}\times i_X$} (m-3-3)
(m-1-3) edge (m-2-2)
(m-3-1) edge node[below] {$t_{n-1}\times 1_X$} (m-3-3)
(m-2-2) edge[commutative diagrams/dashed] node[above right] {$t_{n}$} (m-3-3);
\end{tikzpicture}
\end{center}
where $T^n(i_X)$ is the join of $t_{n-1}\times 1_X$ and $1_{X^n}\times i_X$. 

\colorado{The} next result pieces together \cite[Proposition~26]{DoeHao} and~\cite[Corollary~11]{CalcinesVan}. We have chosen the statement in \thmref{thm:Doe-Hao} \rouge{below} \colorado{for multiple reasons. To start with, the $n$-th sectional fat-wedge and the formulas \colorado{for} all maps appearing in diagram~(\ref{commutative diagram up to homotopy}) \colorado{below} have a simple description} in the case when $i_X:A\toi X$ is a cofibration. \colorado{On} the other hand, we are \colorado{primarily} interested in the case of the diagonal inclusion in $X\times X$. Nonetheless, we remark that the characterization of relative category given by Doeraene and El Haouari in \cite[Proposition 26]{DoeHao} applies for \emph{any} map $i_X:A\to X$.

\begin{theorem}\label{thm:Doe-Hao}
Let $i_X\colon A\toi X$ be a cofibration. We have $\relcat(i_X)\leq n$ if and only if there exits a map $f\colon X\to T^n(i_X)$ \colorado{fitting in a homotopy} commutative diagram
\begin{equation}\label{commutative diagram up to homotopy}
\begin{tikzpicture}[commutative diagrams/every diagram]
\matrix[matrix of math nodes, name=m, commutative diagrams/every cell, column sep=4ex]
    {A&&T^n(i_X)\\
    X&&X^{n+1},\\};
\path[commutative diagrams/.cd, every arrow, every label]
(m-1-1) edge node[above] {$\tau_n$} (m-1-3)
(m-1-1) edge[commutative diagrams/hook] node[left] {$i_X$} (m-2-1)
(m-1-3) edge[commutative diagrams/hook] node[right] {$t_n$} (m-2-3)
(m-2-1) edge[commutative diagrams/dashed] node[above] {$f$} (m-1-3)
(m-2-1) edge node[below] {$\Delta_{n+1}$} (m-2-3);
\end{tikzpicture}
\end{equation}
where $\Delta_{n+1}$ is the diagonal map, $\tau_{n}=\Delta_{n+1}|_A$, and $t_n:T^n(i_X)\toi X^{n+1}$ is the inclusion of the subspace $T^n(i_X)=\{(x_0,\ldots,x_n)\in X^{n+1}\colon x_i\in A \text { for some }i\,\}$.
\end{theorem}

\section{The Fadell-Husseini monoidal topological complexity}\label{sec:FHMTC}
The main goal of this section is to \colorado{complete the proof of} \thmref{Iwase=Dranish=GC}. By \colorado{Remark~\ref{TCM=TCDM=TCFHM} and the work in~\cite{JoseManuel},} the equality $\TC^M(X)=\TC^{\colorado{FH}}_g(X)$ is the only one requiring argumentation. The crux of the proof relies \magenta{on} a Fadell-Husseini type definition of the notions $\relcat_{op}(i_X)$ and $\relcat_g(i_X)$ for a cofibration $i_X:A\toi X$ (see Remarks~\ref{def:relcatop} and~\ref{def:relcatg}).

\begin{definition}\label{def:relcatFH}
Let $i_{X}:A\hookrightarrow X$ be a cofibration. \rouge{A} subset $U$ of $X$ is $A$-relatively sectional if $A\subseteq U$ and there exists a homotopy of pairs $H\colon(U,A)\times [0,1]\to(X,A)$ such that $H(x,0)=x$ and $H(x,1)\in A$ for all $x\in U$. \colorado{The Fadell-Husseini relative category of $i_X$, denoted \rouge{by}} $\relcat^{FH}_{op}(i_X)$, \rouge{is} the least nonnegative integer $n$ such that $X$ admits an open cover $\{U_i\}_{i=0}^n$ satisfying:
\begin{enumerate}
\item $U_{0}$ is $A$-relatively sectional;
\item for $i\geq 1$, $U_i\cap A=\varnothing$ and there are homotopies $H_i\colon U_{i}\times [0,1]\to X$ with $H_{i}(x,0)=x$ and $H_i(x,1)\in A$ for all $x\in U_i$.
\end{enumerate}
If such an integer does not exist, we set $\relcat_{op}^{FH}(i_X)=\infty$.
\end{definition}

\begin{remark}\label{def:relcatop}
If each $U_{i}$ is required to be $A$-relatively sectional in Definition~\ref{def:relcatFH}, we obtain the notion of $\relcat_{op}(i_X)$. The latter concept agrees with the relative category of $i_X$ because of~\cite[Theorem 1.6]{JoseManuel}. In the next proposition we show that \rouge{the equality} $\relcat(i_X)=\relcat_{op}^{FH}(i_X)$ holds as well by following techniques \colorado{similar} to those exposed in~\cite{JoseManuel}.
\end{remark}

\begin{proposition}\label{pro:relcat=relcatopFH}
Let $X$ be a normal space. If $i_{X}\colon A\hookrightarrow X$ is a cofibration, then $\relcat(i_X)=\relcat_{op}^{FH}(i_X)$.
\end{proposition}
\begin{proof}
\colorado{Let $n:=\relcat(i_X)$, which agrees with $\relcat_{op}(i_X)$ in view of~\cite[Theorem 1.6]{JoseManuel}. Choose} an open cover $\{U_0,\ldots,U_n\}$ of $X$ such that, for any $i\geq 0$, $U_i$ is $A$-relatively sectional. Then there are homotopies of pairs $H_i\colon (U_i,A)\times[0,1]\to (X,A)$ such that $H_i(x,0)=x$ and $H_i(x,1)\in A$ for all $x\in U_i$. It is clear that, \colorado{by} setting $U^{*}_{0}:=U_0$, $U^{*}_i:=U_i\setminus A$ and \colorado{by} restricting the homotopies $H_i$ to $U^{*}_i$ for $i\geq 1$, the two items of Definition~\ref{def:relcatFH} are fulfilled. \colorado{Thus $$\relcat^{FH}_{op}(i_X)\leq \relcat(i_X).$$ Next}, in order to prove \colorado{the opposite inequality,} let $\colorado{m:={}}\relcat_{op}^{FH}(i_X)$ and consider an open cover $\{U_i\}_{i=0}^{\colorado{m}}$ of $X$ such that:
\begin{enumerate}
\item $U_{0}$ is $A$-relatively sectional, i.e., $A\subseteq U_0$ and there exists a homotopy of pairs $H_{0}:(U_0,A)\times[0,1]\to (X,A)$ with $H_{0}(x,0)=x$ and $H_0(x,1)\in A$ for all $x\in U_0$;
\item for $i\geq 1$, $U_i\cap A=\varnothing$ and there are homotopies $H_i\colon U_{i}\times [0,1]\to X$ with $H_{i}(x,0)=x$ and $H_i(x,1)\in A$ for all $x\in U_i$.
\end{enumerate}
Since $X$ is a normal space, there are closed sets $A_i$, $B_i$ and open sets $\Theta_i$ \rojo{($i=0,\ldots,\colorado{m}$)} fulfilling \rojo{$A\subseteq A_0$} and $A_i\subseteq\Theta_i\subseteq B_i\subseteq U_i$ \rojo{for all $i$,} with $\{A_i\}_{i=0}^{\colorado{m}}$ \rojo{covering} $X$. Furthermore, there exist Urysohn maps $h_i\colon X\to [0,1]$ such that $h_i(A_i)=\{1\}$ and $h_i(X\setminus\Theta_i)=\{0\}$. For $i\geq 0$, let $L_i\colon (X,A)\times [0,1]\to (X,A)$ be the continuous map defined by 
\[
L_i(x,t)=\begin{cases}
x,&\text{ if }x\in X\setminus B_i;\\
H_i(x,h_i(x)\cdot t),&\text{ if }x\in U_i.
\end{cases}
\]
Observe that $L_i$ is well-defined because $x=H_i(x,h_i(x)\cdot t)$ for all $x\in U_i\setminus B_i$. Likewise, the facts $L_{0}(a,t)=H_{0}(a,t)\in A$ and $L_i(a,t)=a$ for $t\in[0,1]$, $a\in A$ and $i\geq 1$ (recall, $U_i\cap A=\varnothing$ for $i\geq 1$) imply that \red{$L_i$ restricts to a map} $A\times[0,1]\to A$. \colorado{Then} define $L:(X,A)\times[0,1]\to (X^{\colorado{m}+1},T^{\colorado{m}}(i_X))$ \colorado{by} $L:=(L_0,\ldots,L_{\colorado{m}})$. Since $\{A_i\}_{i=0}^{\colorado{m}}$ covers $X$, we get a well-defined map $f:X\to T^{\colorado{m}}(i_X)$ by setting $f(x):=L(x,1)$. Such a map satisfies $L:\Delta_{\colorado{m}+1}\simeq t_{\colorado{m}}\circ f$ and $L|_{A\times [0,1]}:\tau_{\colorado{m}}\simeq f\circ i_X$. Therefore, by Theorem~\ref{thm:Doe-Hao}, $\relcat(i_X)\leq \colorado{\relcat^{FH}_{op}(i_X).}$
\end{proof}

\red{We now} discuss the \colorado{Fadell-Husseini generalized relative category} $\relcat^{FH}_g(i_X)$, which is determined just as in Definition~\ref{def:relcatFH}, \colorado{but without requiring that the covers should consist of open sets.} We show that, under \rouge{mild} hypotheses, \colorado{dropping such a condition is immaterial} (Proposition~\ref{thm:relcat=relcatgFH} below).

\begin{definition}\label{def:relcatgFH}
Let $i_{X}:A\hookrightarrow X$ be a cofibration. We define \colorado{the Fadell-Husseini generalized relative category of $i_X$, denoted by} $\relcat_{g}^{FH}(i_X)$, as the least nonnegative integer $n$ such that $X$ admits a (not necessarily open) \red{cover} $\{U_0,\ldots,U_n\}$ satisfying:
\begin{enumerate}
\item $U_{0}$ is $A$-relatively sectional;
\item for $i\geq 1$, $U_i\cap A=\varnothing$ and there are homotopies $H_i\colon U_{i}\times [0,1]\to X$ with $H_{i}(x,0)=x$ and $H_i(x,1)\in A$ for all $x\in U_i$.
\end{enumerate}
\rouge{We set $\relcat^{FH}_{g}(i_X)=\infty$ if the required coverings do} not exist.
\end{definition}

\begin{remark}\label{def:relcatg}
If each subset $U_i$ \colorado{in Definition~\ref{def:relcatgFH}} is required to be $A$-relative sectional, we obtain the notion of $\relcat_{g}(i_X)$. \red{In view of~\cite[Theorem~2.16]{JoseManuel}}, the latter concept coincides with $\relcat(i_X)$ provided $i_{X}\colon A\toi X$ is a cofibration between $\ANR$ spaces. Paralleling the proof of such a fact, we will prove \colorado{an analogous} conclusion in the \colorado{Fadell-Husseini} context.
\end{remark}
 
Before delving into the equality $\relcat(i_X)=\relcat_{g}^{FH}(i_X)$, we \red{prove} the following technical \colorado{fact} (cf.~\cite[Lemma 2.14]{JoseManuel}):

\begin{lemma}\label{technical lemma}
Let $i_X\colon A\hookrightarrow X$ be a cofibration with $X$ a normal space. Assume that 
\begin{enumerate}
\item $\{U_{0},\ldots,U_n\}$ is an open cover of $X$;
\item $A\subseteq U_{0}$ and there is a homotopy $H_0\colon U_0\times [0,1]\to X$ so that $H_0(x,0)=x$, $H_0(x,1)\in A$, and $H_{0}(-,1)|_A\simeq 1_A$;
\item for $i\geq 1$, $U_i\cap A=\varnothing$ and there are homotopies $H_i\colon U_{i}\times [0,1]\to X$ with $H_{i}(x,0)=x$ and $H_i(x,1)\in A$ for any $x\in U_i$.
\end{enumerate}
\red{Then} $\relcat(i_X)\leq n$.
\end{lemma}
\begin{proof}
The first half \red{of the argument} follows the constructions in the proof of Proposition~\ref{pro:relcat=relcatopFH}: \rojo{Let $A_i, B_i$ be closed sets and $\Theta_i$ be open sets ($i=0,\ldots,n$), with $\{A_i\}_{i=0}^n$ covering $X$, such that $A\subseteq A_0$ and $A_i\subseteq\Theta_i\subseteq B_i\subseteq U_i$ for all $i$.} Choose Urysohn maps $h_i\colon X\to [0,1]$ such that $h_i(A_i)=\{1\}$ and $h_i(X\setminus\Theta_i)=\{0\}$. For $i\geq 0$, \rojo{let} $L_i\colon X\times [0,1]\to X$ \rojo{be defined by}
\[
L_i(x,t)=\begin{cases}
x,&\text{ if }x\in X\setminus B_i;\\
H_i(x,h_i(x)\cdot t),&\text{ if }x\in U_i.
\end{cases}
\]
Set $L:=(L_0,\ldots,L_n)\colon X\times[0,1]\to X^{n+1}$ and note that we still have $L_{i}|_{A\times[0,1]}:A\times[0,1]\to A$, as well as $L:\Delta_{n+1}\simeq t_n\circ f$, where $$\colorado{f:=L(-,1)}\colon X\to T^n(i_X).$$
On the other hand, let $G_0:A\times[0,1]\to A$ be the homotopy between $1_A$ and $H_{0}(-,1)|_A$, this is, $G_0(a,0)=a$ and $G_0(a,1)=H_{0}(a,1)$ for all $a\in A$. Define $G\colon A\times[0,1]\to A^{n+1}\subseteq T^n(i_X)$ to be $G=(G_0,L_{1}|_{A\times [0,1]},\ldots,L_{n}|_{A\times [0,1]})$. Observe that $L_0(a,1)=H_{0}(a,1)=G_{0}(a,1)$, so $G:\tau_n\simeq f\circ i_X$. Therefore, the desired inequality $\relcat(i_X)\leq n$ comes from Theorem~\ref{thm:Doe-Hao}.
\end{proof}

\begin{proposition}\label{thm:relcat=relcatgFH}
Let $i_X\colon A\hookrightarrow X$ be a cofibration between $\ANR$ spaces. We have $\relcat(i_X)=\relcat_{g}^{FH}(i_X)$.
\end{proposition}
\begin{proof}
Clearly $\relcat_{g}^{FH}(i_X)\leq \relcat_{op}^{FH}(i_X)=\relcat(i_X)$, where the latter relation holds in view of Proposition~\ref{pro:relcat=relcatopFH} (recall \magenta{that} $X$ is a normal space since it is metrizable). \red{We} show the inequality $\relcat(i_X)\leq\relcat_{g}^{FH}(i_X)$.

\red{Let} $n:=\relcat_{g}^{FH}(i_X)$ and consider a (not necessarily open) \red{cover} $\{U_i\}_{i=0}^{n}$ of $X$ such that:
\begin{enumerate}
\item $U_{0}$ is $A$-relatively sectional, i.e., $A\subseteq U_0$ and there exists a homotopy of pairs $H_{0}:(U_0,A)\times[0,1]\to (X,A)$ with $H_{0}(x,0)=x$ and $H_0(x,1)\in A$ for all $x\in U_0$;
\item for $i\geq 1$, $U_i\cap A=\varnothing$ and there are homotopies $H_i\colon U_{i}\times [0,1]\to X$ with $H_{i}(x,0)=x$ and $H_i(x,1)\in A$ for all $x\in U_i$.
\end{enumerate}
The argument below for $i=0$ is the one in the proof of~\cite[Theorem 2.16]{JoseManuel}. We review the details since we will then describe a slight variant in order to deal with the case of $i>0$. Consider the following factorization of $i_{X}$ through its mapping cocylinder:
\begin{center}
\begin{tikzpicture}[commutative diagrams/every diagram]
\matrix[matrix of math nodes, name=m, commutative diagrams/every cell, column sep=4ex]
    {A&&X\\
    &\widehat{A},&\\};
\path[commutative diagrams/.cd, every arrow, every label]
(m-1-1) edge[commutative diagrams/hook] node[above] {$i_X$} (m-1-3)
(m-1-1) edge node[below left] {$j$} (m-2-2)
(m-2-2) edge node[below right] {$p$} (m-1-3);
\end{tikzpicture}
\end{center}
where $\widehat{A}=\{(a,\beta)\in A\times X^{[0,1]}\colon i_X(a)=\beta(0)\}$, $p$ is a fibration and $j$ a homotopy equivalence. As observed in~\cite[Lemma 2.13]{JoseManuel}, $\widehat{A}$ is also an $\ANR$. Define $s_0\colon U_0\to\widehat{A}$ to be $s_0=j\circ H_{0}(-,1)$, then $p\circ s_0\simeq U_0\toi X$ and $s_{0}|_A\simeq j$. Actually, since $p$ is a fibration, there is no problem in assuming that $p\circ s_0=U_0\toi X$. Following the proof of~\cite[Theorem 2.7]{JoseManuel}, there exist an open neighborhood $V_{0}$ of $U_0$ in $X$ and a map $\sigma_0\colon V_{0}\to \widehat{A}$ such that $p\circ\sigma_0=V_0\toi X$ and $\sigma_{0}|_{U_0}\simeq s_0$. In particular, \begin{equation}
\sigma_{0}|_A=(\sigma_0|_{U_0})|_A\simeq s_{0}|_{A}\simeq j.
\end{equation}
If $j':\widehat{A}\to A$ denotes a homotopy inverse of $j$, then 
$$
i_X\circ j'\circ\sigma_{0}=p\circ j\circ j'\circ \sigma_0\simeq p\circ \sigma_{0}=V_0\toi X.
$$
Hence, there exists a homotopy $G_0:V_0\times[0,1]\to X$ such that $G_{0}(x,0)=x$, $G_0(x,1)=i_X\circ j'\circ\sigma_{0}(x)\in A$ and $G_{0}(-,1)|_A=j'\circ (\sigma_{0}|_{A})\simeq j'\circ j\colorado{{}\simeq{}} 1_A$. 

On the other hand, for $i\geq 1$, set $s_i:=j\circ H_i(-,1)\colon U_i\to\widehat{A}$. An examination of the proof above (omitting \red{those} steps \red{that involve} $A\subseteq U_0$) reveals that there are open neighborhoods $V_i$ of $U_i$ in $X$ together with maps $\sigma_i:V_i\to \widehat{A}$ so that $p\circ\sigma_i=V_i\toi X$ and $\sigma_i|_{U_i}\simeq s_i$. Without losing generality we may assume that $V_i\cap A=\varnothing$ for, otherwise, we simply set $V'_i:=V_{i}\setminus A$ and $\sigma'_i=\sigma_i|_{V'_i}$. Furthermore, we have homotopies $G_i\colon V_i\times[0,1]\to X$ such that $G_i(x,0)=x$ and $G_i(x,1)=i_X\circ j'\circ\sigma_i(x)\in A$ for all $x\in V_i$.

\colorado{The required inequality} $\relcat(i_X)\leq n$ \colorado{then follows from} Lemma~\ref{technical lemma}.
\end{proof}

\colorado{We complete the proof of \thmref{Iwase=Dranish=GC} by} bringing together the main results of this section.

\begin{proof}[Proof of \thmref{Iwase=Dranish=GC}]
Since $X$ is an $\ANR$ space, the canonical embedding $i_X\colon \Delta X\toi X\times X$ is a cofibration between $\ANR$ spaces, \colorado{so that} $\relcat(i_X)=\relcat_{op}(i_X)=\relcat_g(i_X)$, \colorado{by~\cite[Theorems 1.6 and 2.16]{JoseManuel}. Furthermore,} the equalities $\relcat(i_X)=\relcat_{op}^{FH}(i_X)=\relcat^{FH}_g(i_X)$ come from Propositions~\ref{pro:relcat=relcatopFH} and~\ref{thm:relcat=relcatgFH}.

In order to show \colorado{the equality} $\TC^M(X)=\TC^{\colorado{FH}}_g(X)$, \colorado{we first} note that $\relcat_{g}^{FH}(i_X)\leq\TC^{\colorado{FH}}_g(X)$. \colorado{The latter} fact comes by noticing that, if $\{U_i\}_{i=0}^n$ is a (not necessarily open) \red{cover} of $X\times X$ and $s_i\colon U_i\to X^{[0,1]}$ are the local sections of the fibration \colorado{$e_{01}$} coming from Definition~\ref{def:FHGMTC}, then one can define homotopies $H_{0}\colon (U_{0},\Delta X)\times [0,1]\to (X\times X,\Delta X)$ and $H_i\colon U_i\times[0,1]\to X\times X$ ($i\geq 1$) as $H_i(x,y,t)=(s_{i}(x,y)(t),y)$ ($i\geq 0$) satisfying the two items of Definition~\ref{def:relcatgFH}. Likewise, it is clear that $\TC^{\colorado{FH}}_g(X)\leq\TC^{\colorado{FH}}(X)$; nevertheless, the latter expression agrees with $\TC^{M}(X)$ due to Remark~\ref{TCM=TCDM=TCFHM}. Finally, \colorado{the equality} $\TC^M(X)=\relcat(i_X)$ follows from~\cite[Theorem 12]{CGV}, while $\relcat(i_X)$ equals $\relcat_{g}^{FH}(i_X)$ by our initial discussion. In summary, \colorado{we have}
$$
\relcat_{g}^{FH}(i_X)\leq\TC^{\colorado{FH}}_g(X)\leq \TC^{\colorado{FH}}(X)=\TC^{M}(X)=\relcat_{g}^{FH}(i_X),
$$
\colorado{which completes the proof.}
\end{proof}

\section{Proof of Theorem~\ref{dejosemanuel}}
\colorados{We \rouge{start in the non-generalized setting, i.e., by proving that} the stasis condition~$(2)$ in Definition~\ref{def:FHGMTC} can be omitted without altering the numerical value of $\TC^{FH}(X)$. \rouge{Let} $\{(U_i,s_i)\}_{i=0}^n$ be a motion planner \rouge{with} $\Delta X\subseteq U_0$ and $\Delta X\cap U_i=\varnothing$ for all $i\geq 1$. \rouge{We do not assume} that the section $s_0$ of $e_{01}$ yields constant paths when restricted to $\Delta X$, \rouge{but we do assume that all subsets $U_i$ are open. The task is} to construct a motion planner $\{(V_i,\sigma_i)\}_{i=0}^n$ \rouge{of a Fadell-Husseini type,} that is, \rouge{one that consists of open sets $V_i$ satisfying} $\Delta X\cap V_i=\varnothing$ for all $i\geq 1$, \rouge{as well as} $\Delta X\subseteq V_0$ \rouge{with} $\sigma_0(x,x)=c_x$, the constant path at $x$, for all $x\in X$.}

\rouge{If $n=0$, then $X$ is in fact contractible, so that the homotopy invariance of the monoidal topological complexity for locally equiconnected spaces (\cite[Proposition~2.17]{JoseManuel}) gives the required motion planner of a Fadell-Husseini type. We can therefore assume $n\geq1$. By~\cite[Theorem~II.1]{DyEi}, there is} \colorados{an open neighborhood $W$ of $\Delta X$ in $X\times X$ and a local section $\lambda:W\to X^{[0,1]}$ \rouge{of the end-points evaluation map $e_{01}$ satisfying} $\lambda(x,x)=c_x$ for all $x\in X$. Furthermore, by the normality \rouge{assumption,} there is an open cover $\{W_i\}_{i=0}^n$ of $X\times X$ such that $W_i\subseteq\overline{W_i}\subseteq U_i$ for all $i\geq 0$. \rouge{Consider the} open neighborhood $N$ of $\Delta X$ \rouge{given} by}
$$
\rouge{N=W\cap U_0\cap\left((X\times X)\setminus\overline{W_1}\rule{0mm}{4mm}\right)\cap\cdots\cap\left((X\times X)\setminus\overline{W_n}\rule{0mm}{4mm}\right).}
$$
\colorados{\rouge{Using once more} the normality of $X\times X$, take an open set $M$ in $X\times X$ \rouge{with} $\Delta X\subseteq M\subseteq\overline{M}\subseteq N$. Let $V_0=(U_0\setminus\overline{M})\rouge{{}\sqcup{}} M$ and define the reserved section $\sigma_0:V_0\to X^{[0,1]}$ of $e_{01}$ by
\[
\sigma_0(x,x')=\begin{cases}
s_0(x,x'),&\text{ if }(x,x')\in U_0\setminus\overline{M};\\
\lambda(x,x'),&\text{ if }(x,x')\in M.
\end{cases}
\]
Lastly, for \rouge{$1\leq i\leq n$,} set $V_i=W_i\rouge{{}\sqcup{}} (N\setminus\Delta X)$ and define the local sections $\sigma_i:V_i\to X^{[0,1]}$ of $e_{01}$ by
\[
\sigma_i(x,x')=\begin{cases}
s_i(x,x'),&\text{ if }(x,x')\in W_i;\\
\lambda(x,x'),&\text{ if }(x,x')\in N\setminus\Delta X.
\end{cases}
\]
\rouge{Then} $\{(V_i,\sigma_i)\}_{i=0}^n$ \rouge{is the required motion planner of Fadell-Husseini type, for $\{V_0,V_1,\ldots,V_n\}$ covers $X\times X$. Indeed, since $W_i\subseteq V_i$ for $i\geq 1$, the covering assertion follows by observing that $W_0\setminus\left(V_1\cup\cdots\cup V_n\right)\subseteq V_0$:}}
\begin{align*}
\rouge{W_0\setminus{}}&\rouge{\left(V_1\cup\cdots\cup V_n\right)\subseteq U_0\setminus\left(V_1\cup\cdots\cup V_n\right)=\left((U_0\setminus\overline{M})\cup(\overline{M}\setminus M)\cup M\right)\setminus\left(V_1\cup\cdots\cup V_n\right)}\\
&\rouge{=\left(V_0\cup(\overline{M}\setminus M)\right)\setminus\left(V_1\cup\cdots\cup V_n\right)=\left(V_0\setminus\left(V_1\cup\cdots\cup V_n\right)\right) \cup \left( 
\overline{M}\setminus \left(M\cup V_1\cup\cdots\cup V_n\right)
\right),}
\end{align*}
\rouge{where $\overline{M}\setminus \left(M\cup V_1\cup\cdots\cup V_n\right)\subseteq N\setminus\left(\Delta X\cup V_1\cup\cdots\cup V_n\right)=\left(N\setminus\Delta X\right)\setminus\left(V_1\cup\cdots\cup V_n\right)=\varnothing$, as $n\geq1$.}

\medskip
\rouge{We \magenta{end up by sketching} the argument for} the generalized case. \rouge{Let} $\{(U_i,s_i)\}_{i=0}^n$ be a \rouge{generalized} motion planner \rouge{consisting of a cover $\{U_i\}_{i=0}^n$ of $X\times X$ by not necessarily open} subsets $U_i$ \rouge{such that} $\Delta X\subseteq U_0$ and $\Delta X\cap U_i=\varnothing$ for all $i\geq 1$, \rouge{and of sections $s_i\colon U_i\to X^{[0,1]}$ of $e_{0,1}$}. \rouge{Again, without assuming} that~$s_0$ is a reserved section of $e_{01}$, \rouge{the task is to assure the existence of a corresponding generalized motion planner, one of whose rules is defined on the whole diagonal via constant paths.} \Rojo{Following the proof of~\cite[Theorem~2.7]{JoseManuel}, we can construct a new motion planner $\{(V_i,\sigma_i)\}_{i=0}^n$ so that, for all $i\geq 0$, $V_i$ is an open subset of $X\times X$, $U_{i}\subseteq V_i$ (so $\Delta X\subseteq V_0$), and $\sigma_i|_{U_i}\simeq s_i$. Furthermore, without loss of generality we can assume $\Delta X\cap V_i=\varnothing$ for all $i\geq 1$}. \rouge{Then}, by the \rouge{argument in the} non-generalized case, we can \rouge{fix the situation so to have in addition} $\sigma_0(x,x)=c_x$, the constant path at $x$, for all $x\in X$, \magenta{thus completing} the argument.

\section{Polyhedral products of groups}\label{applications}

This section is devoted to proving \thmref{TC(G^K)=TCM(G^K)} via \thmref{Iwase=Dranish=GC}. We start by giving a quick overview on polyhedral products and discussing \magenta{a} direct consequence \magenta{(Corollary~\ref{pparentH})} of \thmref{TC(G^K)=TCM(G^K)}. 
\begin{definition}
Let $(\underline{X},\underline{A})=\{(X_{i},A_{i})\}_{i=1}^{m}$ be a family of pairs of spaces and $K$ be an abstract simplicial complex with $m$ vertices labeled by the set $\{1,2,\ldots,m\}$. The polyhedral product determined by $(\underline{X},\underline{A})$ and $K$ \red{is}\footnote{Note that $(\underline{X},\underline{A})^{\sigma_1}$ is contained in $(\underline{X},\underline{A})^{\sigma_2}$ provided $\sigma_1\subseteq\sigma_2$. Therefore, it suffices to take the union over all the maximal simplices of $K$ in~(\ref{definition of polyhedral product}), this is, simplices that are not contained in any other simplex of $K$.} 
\begin{equation}\label{definition of polyhedral product}
(\underline{X},\underline{A})^{K}=\cup_{\sigma\in K}(\underline{X},\underline{A})^{\sigma},
\end{equation}
where $(\underline{X},\underline{A})^{\sigma}=\prod_{i=1}^{m}Y_{i}$ and  
\[
Y_{i}=\begin{cases}
A_i,&\text{if}\hspace{1mm} i\in\{1,2,\ldots,m\}\setminus\sigma;\\ 
X_{i},&\text{if}\hspace{1mm} i\in \sigma.\
\end{cases}
\]
\end{definition}

We are interested in the case where all $A_{i}=\ast$, \rouge{in which case} $(\underline{X},\ast)^{K}$ and $(\underline{X},\ast)^{\sigma}$ are simply denoted by $\underline{X}^{K}$ and $\underline{X}^{\sigma}$, respectively. Moreover, it is clear that, for any $\sigma\in K$, $\underline{X}^\sigma$ is a retract of $\underline{X}^K$ and $\underline{X}^\sigma\approx\prod_{i\in\sigma}X_i$.

\begin{definition}\label{definition of LS logarithmic}
A family of based spaces $\underline{X}=\{(X_{i},\ast)\}_{i=1}^{m}$ is said to be $\LS$-logarithmic if
$$
\cat(X_{i_1}\times\cdots\times X_{i_k})=\cat(X_{i_1})+\cdots+\cat(X_{i_k})
$$
holds \red{for all strictly increasing sequences $1\leq i_1<\cdots<i_k\leq m$.}
\end{definition}
\begin{example}
The family $\underline{G}=\{(U(n),e_i)\}_{i=1}^{m}$, where $U(n)$ denotes the \red{$n$-th} unitary group, fulfills the requirements of Theorem~\ref{TC(G^K)=TCM(G^K)}. The $\LS$-logarithmicity hypothesis comes from~\cite[Example 3.3]{Lupton}, while $\cat(U(n))=n$ is guaranteed by~\cite[Theorem 9.47]{CLOT}. \thmref{TC(G^K)=TCM(G^K)} thus gives
\begin{align*}
\TC(\polyg)&=\TC^M(\polyg)=\max\bigg\{\sum_{i\in\sigma_{1}\cup\sigma_{2}}\cat(U(n))\colon \sigma_1,\sigma_2\in K\bigg\}\\
&=\max\bigg\{\sum_{i\in\sigma_{1}\cup\sigma_{2}}n\colon \sigma_1,\sigma_2\in K\bigg\}=n\cdot\max\{|\sigma_{1}\cup\sigma_{2}|\colon \sigma_1,\sigma_2\in K\}.
\end{align*}
Setting $n=1$, we recover the result obtained in~\cite[Theorem 2.7]{GGY2016polyhedral} (for $r=2$ and \red{all spheres being 1-dimensional}). In fact, Theorem~\ref{TC(G^K)=TCM(G^K)} 
determines the topological complexity and the monoidal topological complexity of a polyhedral product whose factors are unitary groups or special unitary groups of possibly different dimensions. In such a case, the $\LS$-logarithmicity hypothesis is guaranteed by~\cite[Example 3.3]{Lupton}.
\end{example}
\begin{corollary}\label{pparentH}
Let $\underline{G}$ \red{be an LS-logarithmic} based family \red{as the one} in \thmref{TC(G^K)=TCM(G^K)}. If \red{no} $G_{i}$ is contractible, then $\polyg$ admits an $H$-space structure if and only if $K$ is the standard $(m-1)$-simplex.
\end{corollary}
\begin{proof}
If $K$ is the standard $(m-1)$-simplex, then $\polyg=G_{1}\times\cdots\times G_{m}$ is a topological group, and hence it is an $H$-space. On the other hand, suppose that $\polyg$ admits \red{an $H$-space structure. Being} connected \red{and cellular, $\polyg$ satisfies} 
\begin{align*}
&\max\bigg\{\sum_{i\in\sigma_{1}\cup\sigma_{2}}\cat(G_{i})\colon \sigma_1,\sigma_2\in K\bigg\}=\TC(\polyg)=\cat(\polyg)\\
&=\max\bigg\{\sum_{i\in\sigma}\cat(G_{i})\colon\sigma\in K\bigg\},
\end{align*}
where the first equality comes from Theorem~\ref{TC(G^K)=TCM(G^K)}, the second one follows from~\cite[Theorem 1]{Lupton}, and the third one is guaranteed by~(\ref{cat of G^K}). Finally, bearing in mind that both maximums above agree, and \red{that} $\cat(G_{i})\geq 1$ for all $i\in\{1,\ldots,m\}$, we conclude that $K$ is the standard $(m-1)$-simplex.
\end{proof}

\rouge{We now delve into the proof of \thmref{TC(G^K)=TCM(G^K)}, starting with the following auxiliary result:} 

\begin{lemma}\label{Lower bound}
Let \rouge{$\underline{G}^K$} be as in Theorem~\ref{TC(G^K)=TCM(G^K)}. Then \rouge{$\TC(\polyg)$ is no less than}
\begin{equation}\label{definition of C(G1,Gm,K)}
\colorado{\normados:={}}\max\bigg\{\sum_{i\in\sigma_{1}\cup\sigma_{2}}\cat(G_{i})\colon \sigma_1,\sigma_2\in K\bigg\}.
\end{equation}
\end{lemma}

\rouge{Note that the maximum in~(\ref{definition of C(G1,Gm,K)}) is realized by maximal simplices of $K$.}

\begin{proof}
From~\cite[Corollary 6.15]{AGO} we get $\TC(\polyg)\geq\cat(\und{G}^{\sigma_1}\times\und{G}^{\sigma_2})$ for any disjoint simplices $\sigma_1,\sigma_2\in K$. The result follows in view of the $\LS$-logarithmicity hypothesis.
\end{proof}

Since $\TC(\polyg)\leq\TC^M(\polyg)$, the proof of \thmref{TC(G^K)=TCM(G^K)} will be complete once we prove:

\begin{proposition}\label{upper bound}
Let $\rouge{\underline{G}^K}$ be as in Theorem~\ref{TC(G^K)=TCM(G^K)}. \rouge{Then} $\TC^M(\polyg)\leq\rouge{\normados.}$ 
\end{proposition}

\subsection{Proof of Proposition~\ref{upper bound}}
As we remarked in the proof of \propref{MTC of a CW-TG}, \colorados{locally compact $\CW$ complexes are $\ANR$ spaces}. Consequently, $\underline{G}^{\sigma}$ is an $\ANR$ for each $\sigma\in K$, and therefore $\polyg=\cup_{\sigma\in K}\underline{G}^{\sigma}$ is an $\ANR$ as well. Furthermore, in view of \thmref{Iwase=Dranish=GC}, it suffices to show that $\TC^{\colorado{FH}}_g(\polyg)\leq\normados$.

In order to attain the latter task, we will construct a general cover (not necessarily open) of $\polyg\times \polyg$ fulfilling the conditions of \defref{def:FHGMTC} (Proposition~\ref{general cover of G^KXG^K} below).

For each $i\in\{1,\ldots,m\}$, let $\{V_{i0},\ldots,V_{ic_{i}}\}$ be an open cover of $G_{i}\times G_{i}$ together with reserved sections $\lambda_{ik}\colon V_{ik}\rightarrow G_{i}^{[0,1]}$ of the end-points evaluation map $\colorado{e_{01}}\colon G_{i}^{[0,1]}\rightarrow G_{i}\times G_{i}$. Here, $c_i$ stands for the $\LS$ category of the corresponding polyhedral product factor $G_i$. As shown in the proof of Proposition~\ref{MTC of a CW-TG}, we can assume that the diagonal of $G_{i}$ is contained in $V_{i0}$, as well as $\Delta G_{i}\cap V_{ik}=\varnothing$ for all $k\in\{1,\ldots,c_i\}$. 

The open sets $V_{ik}$ might not be disjoint; however, this requirement can be achieved by redefining 
$$
U_{ik}=V_{ik}\setminus (V_{i0}\cup\cdots\cup V_{i(k-1)})
$$
for each $k\in\{0,1,\ldots,c_i\}$ and $i\in\{1,\ldots,m\}$ (so $U_{i0}=V_{i0}$). We say that a pair $(a,b)$ in $G_{i}\times G_{i}$ produces $k$ closed conditions if $(a,b)\in U_{ik}$, with $k\in\{0,1,\ldots,c_i\}$. The number of closed conditions produced by $(a,b)$ is denoted by $C(a,b)$.

We regard an element $(a_1,a_2)$ of $\polyg\times\polyg$ as a matrix of size $m\times 2$, i.e., 
\begin{equation}
\nonumber
(a_1,a_2)=
\begin{pmatrix}
a_{11}&a_{12}\\
\vdots&\vdots\\
a_{m1}&a_{m2}
\end{pmatrix},
\end{equation}
where each column belongs to $\polyg$, say 
$$
(a_{11},\ldots,a_{m1})\in \underline{G}^{\sigma_{1}} \text{ and } (a_{12},\ldots,a_{m2})\in \underline{G}^{\sigma_{2}},
$$ 
with $\sigma_{1},\sigma_{2}\in K$. We know that each row $(a_{i1},a_{i2})$ of the matrix $(a_1,a_2)$ lies in a unique set $U_{ik}$ for $k=C(a_{i1},a_{i2})\in\{0,1,\ldots,c_i\}$, so the number of closed conditions determined by $(a_1,a_2)$, denoted \rouge{by} $C(a_1,a_2)$, is defined to be the sum of closed conditions produced by the rows $(a_{i1},a_{i2})$ of $(a_1,a_2)$, this is, 
$$
C(a_1,a_2):=\sum_{i=1}^{m}C(a_{i1},a_{i2}).
$$
It is clear that $C(a_1,a_2)\leq\sum_{i\in\sigma_{1}\cup\sigma_{2}}c_i\leq \normados$, and hence we have proved: 

\begin{proposition}\label{general cover of G^KXG^K}
The sets $W_{j}=\{(a_1,a_2)\in \polyg\times\polyg\colon C(a_1,a_2)=j\}$, with $j$ belonging to $\{0,1,\ldots,\normados\}$, form a pairwise disjoint cover of $\polyg\times\polyg$.
\end{proposition}

The proof of Proposition~\ref{upper bound} will be complete once a \magenta{(continuous)} local rule is constructed on each $W_j$. \magenta{The} task is attained by splitting $W_j$ into topological disjoint subsets (see Proposition~\ref{topological union proof} below), and then defining a \magenta{(continuous)} local section of the fibration \colorado{$e_{01}$} on each one of them.

A partition of $j$, with $j\in\{0,1,\ldots,\normados\}$, is an ordered tuple $(j_1,\ldots, j_{m})$ of nonnegative integers such that $j=j_{1}+\cdots+j_{m}$ and $0\leq j_{i}\leq c_i$ for each $i\in\{1,\ldots,m\}$. For such a partition of $j$, set
$$
W_{(j_1,\ldots, j_{m})}=\{(a_1,a_2)\in \polyg\times\polyg\colon C(a_{i1},a_{i2})=j_{i},\text{ for }i\in\{1,\ldots,m\}\}.
$$

It is straightforward to see that
\begin{equation}\label{topological union}
W_{j}=\displaystyle\bigsqcup_{(j_1,\ldots, j_{m})}W_{(j_1,\ldots, j_{m})},
\end{equation}
where the disjoint union runs over all partitions of $j$. We next show that~(\ref{topological union}) is in fact a topological union, this is, $W_{j}$ has the weak topology determined by the several $W_{(j_1,\ldots, j_{m})}$.

\begin{proposition}\label{topological union proof}
Let $j\in\{0,1,\ldots,\normados\}$. If $(j_1,\ldots, j_{m})$ and $(r_1,\ldots, r_{m})$ are two different partitions of $j$, then
\begin{equation}
\nonumber
\overline{W_{(j_1,\ldots, j_{m})}}\cap W_{(r_1,\ldots, r_{m})}=\varnothing=W_{(j_1,\ldots, j_{m})}\cap\overline{W_{(r_1,\ldots, r_{m})}}.
\end{equation}
\end{proposition}
\begin{proof}
Since $(j_1,\ldots,j_{m})\neq (r_1,\ldots,r_{m})$, there is a natural number $\ell\in\{1,\ldots,m\}$ with $j_{\ell}\neq r_{\ell}$, say $j_{\ell}<r_{\ell}$, while the equality $j_{1}+\cdots+j_{m}=j=r_{1}+\cdots+r_m$ forces the existence of another natural number $q\in\{1,\ldots,m\}$ such that $r_{q}<j_{q}$.

For elements $(a_{1},a_{2})\in W_{(j_1,\ldots,j_{m})}$ and $(b_{1},b_{2})\in W_{(r_1,\ldots,r_{m})}$ we have 
$$
(a_{\ell 1},a_{\ell 2})\in U_{\ell j_{\ell}}\text{ and }(b_{\ell 1},b_{\ell 2})\in U_{\ell r_{\ell}},
$$
then $(a_{\ell 1},a_{\ell 2})\in V_{\ell j_\ell}$ and $(b_{\ell 1},b_{\ell 2})\notin V_{\ell j_\ell}$ since $j_{\ell}<r_{\ell}$. It is clear that the latter condition is inherited by elements of $\overline{W_{(r_1,\ldots,r_{m})}}$, so the second equality of the proposition follows.

The statement $\overline{W_{(j_1,\ldots,j_{m})}}\cap W_{(r_1,\ldots,r_{m})}=\varnothing$ follows by using the condition $r_{q}<j_{q}$.
\end{proof}

In the rest of the section we construct a continuous local section of the end-point\Rojo{s}
evaluation map \colorado{$e_{01}$} on each $W_{(j_1,\ldots, j_{m})}$.  Such a task is performed in the following way: For $i\in\{1,\ldots,m\}$, let $d_i$ denote a metric on $G_i$. \colorados{Since $d_i$ is always equivalent to a bounded metric, we can assume that the diameter of $G_{i}$, defined by $\delta(G_i)=\sup\{d_{i}(a,b)\colon a,b\in G_{i}\}$, is finite.} Likewise, there is no problem in assuming that each diameter $\delta(G_i)$ is positive.

We reparametrize the initial navigational instructions $\lambda_{ik}$ as follows: For $k\in\{0,1,\ldots,c_i\}$, consider the section $\tau_{ik}\colon U_{ik}\rightarrow G_{i}^{[0,1]}$ of the end-points evaluation map $\colorado{e_{01}}\colon G_{i}^{[0,1]}\rightarrow G_{i}\times G_{i}$ where, for $(a_1,a_2)\in U_{ik}$, 
\[
\tau_{ik}(a_{1},a_{2})(t)=\begin{cases}
a_1,&\text{ if } d_{i1}+d_{i2}=0;\\
\lambda_{ik}(a_1,a_2)\big(\frac{2\delta(G_i)t}{d_{i1}+d_{i2}}\big),&\text{ if } 0\leq t\leq \frac{d_{i1}+d_{i2}}{2\delta(G_i)}\neq 0;\\
a_{2},&\text{ if } 0\neq \frac{d_{i1}+d_{i2}}{2\delta(G_i)}\leq t\leq 1;
\end{cases}
\]
and $d_{ij}=d_{i}(a_j,e_{i})$, $j=1,2$. Recall that $e_{i}$ denotes the neutral element of $G_{i}$.

The path $\tau_{ik}$ is clearly continuous on the open subset of $U_{ik}$ determined by the condition $d_{i1}+d_{i2}\neq 0$. The latter open subset of $U_{ik}$ equals in fact $U_{ik}$ unless $k=0$, so that $\tau_{ik}$ is continuous on the whole $U_{ik}$ for $k\in\{1,\ldots,c_i\}$. The continuity of $\tau_{i0}$ on $U_{i0}$ follows from the continuity of the reserved section $\lambda_{i0}$.

Define the (not necessarily continuous) section 
$$
\varphi\colon\polyg\times\polyg\rightarrow (\prod_{i=1}^{m}G_{i})^{[0,1]}
$$
of the fibration \colorado{$e_{01}$} to be $\varphi(a_{1},a_{2})=(\varphi_{1}(a_{11},a_{12}),\ldots,\varphi_{m}(a_{m1},a_{m2})),$
whose $i$th coordinate $\varphi_{i}(a_{i1},a_{i2})$ is the path in $G_{i}$, from $a_{i1}$ to $a_{i2}$, given by
\begin{equation}\label{final movement 2}
\varphi_{i}(a_{i1},a_{i2})(t)=\begin{cases}
a_{i1},&\text{ if } 0\leq t\leq t_{a_{i1}};\\
\mu(a_{i1},a_{i2})(t-t_{a_{i1}}),&\text{ if } t_{a_{i1}}\leq t\leq 1.
\end{cases}
\end{equation}
Here, $t_{a_{i1}}=\frac{1}{2}-\frac{d_{i}(a_{i1},e_{i})}{2\delta(G_i)}$ and 
\begin{equation}\label{final formulas 2}
\mu(a_{i1},a_{i2})=\begin{cases}
\tau_{i0}(a_{i1},a_{i2}),&\text{ if } (a_{i1},a_{i2})\in U_{i0};\\
\hspace{1mm}\vdots&\hspace{2mm}\vdots\\
\tau_{ic_{i}}(a_{i1},a_{i2}),&\text{ if } (a_{i1},a_{i2})\in U_{ic_{i}}.\\
\end{cases}
\end{equation}
By construction, the map $\varphi$ is a section of the fibration
$$
\colorado{e_{01}}\colon(\prod_{i=1}^{m}G_i)^{[0,1]}\rightarrow \prod_{i=1}^{m}G_i\times\prod_{i=1}^{m}G_i.
$$

Although $\varphi$ fails to be a continuous global section of \colorado{$e_{01}$,} its restriction to each $W_{(j_1,\ldots, j_{m})}$, where $(j_1,\ldots,j_{m})$ is a partition of $0\leq j\leq\normados$, is continuous since formulas~(\ref{final formulas 2}) can be rewritten as
\begin{equation}
\nonumber
\mu=\begin{cases}
\tau_{i0},&\text{ if } j_{i}=0;\\
\hspace{1mm}\vdots&\hspace{2mm}\vdots\\
\tau_{ic_{i}},&\text{ if } j_{i}=c_{i}.
\end{cases}
\end{equation}

\begin{remark}\label{unraveling formulas}
In preparation for the proof of our final result, we spell out formulas~(\ref{final movement 2}) in order to unravel the motion provided by $\varphi$ at the level of each polyhedral product factor $G_i$. Concretely, if $(a_{i1},a_{i2})\in U_{ik}$ for some $k\in\{0,1,\ldots,c_i\}$, the path $\varphi_{i}(a_{i1},a_{i2})$ is described as follows:
\begin{itemize}
\item if $0\leq t\leq \frac{1}{2}-\frac{d_{i}(a_{i1},e_{i})}{2\delta(G_i)}$, then stay at $a_{i1}$; 
\item if $\frac{1}{2}-\frac{d_{i}(a_{i1},e_{i})}{2\delta(G_i)}\leq t\leq \frac{1}{2}+\frac{d_{i}(a_{i2},e_{i})}{2\delta(G_i)}$, then move from $a_{i1}$ to $a_{i2}$ at constant speed via $\tau_{ik}$;
\item if $\frac{1}{2}+\frac{d_{i}(a_{i2},e_{i})}{2\delta(G_i)}\leq t\leq 1$, then stay at $a_{i2}$. 
\end{itemize} 
\end{remark}

Recall that $\varphi$ was defined \magenta{as a map} from $\polyg\times\polyg$ to $\left(\prod_{i=1}^m G_i\right)^{[0,1]}$. \magenta{We} next show that \magenta{the image of $\varphi$ is contained} in $(\polyg)^{[0,1]}$, thus completing the proof of Proposition~\ref{upper bound}.
\begin{proposition}
The image of $\varphi$ is contained in $(\polyg)^{[0,1]}$.
\end{proposition} 
\begin{proof}
Let $(a_{1},a_{2})\in \polyg\times\polyg$, we need to prove that $\varphi(a_1,a_{2})([0,1])\subseteq \polyg$. \colorado{So,} assume $(a_{11},\ldots,a_{m1})\in \underline{G}^{\sigma_{1}}$ and $(a_{12},\ldots,a_{m2})\in \underline{G}^{\sigma_{2}}$ with $\sigma_1,\sigma_{2}\in K$. By Remark~\ref{unraveling formulas}, for all $i\notin\sigma_{1}$, $a_{i1}=e_i$ keeps its position through time $t\leq 1/2$, so that $\varphi(a_1,a_{2})([0,1/2])\subseteq \underline{G}^{\sigma_1}\subseteq \polyg$. Again, Remark~\ref{unraveling formulas} shows that, for $i\notin\sigma_{2}$, the path $\varphi_{i}(a_{i1},a_{i2})$ has reached its final position $a_{i2}=e_i$ at time $1/2$, so that $\varphi(a_1,a_{2})([1/2,1])\subseteq \underline{G}^{\sigma_{2}}\subseteq \polyg$, and the proof is complete.
\end{proof}


\begin{thebibliography}{CLOT}
\bibitem{AGO}
Jorge Aguilar-Guzm\'an, Jes\'us Gonz\'alez, and John Oprea, 
\newblock{\em Right-angled Artin groups, Polyhedral products and the $\TC$-generating function}, \newblock Proceedings of the Royal Society of Edinburgh (accepted for publication).
\bibitem{CGV}
J.G. Carrasquel-Vera, J.M. Garc\'ia-Calcines, and L. Vandembroucq, 
\newblock{\em Relative category and monoidal topological complexity,}
\newblock Topology and its Applications {\bf 171} (2014), 41--53.
\bibitem{CLOT}
Octav Cornea, Gregory Lupton, John Oprea, and Daniel Tanr\'e,
\newblock{\em Lusternik-Schnirelmann category},
\newblock Mathematical Surveys and Monographs, vol. 103, American Mathematical Society, Providence, 2003.
\bibitem{DoeHao}
J.M.~Doeraene and M.~El Haouari,
\emph{Up-to-one approximations of sectional category and topological complexity},
Topology and its Applications {\bf 265} (2019), no. 5, 766--783.
\bibitem{Dran}
Alexander Dranishnikov,
\newblock {\em Topological complexity of wedges and covering maps},
\newblock \red{Proceedings of the American Mathematical Society} {\bf 142} (2014), no. 12, 4365--4376.
\bibitem{DyEi}
Eldon Dyer and Samuel Eilenberg,
\newblock {\em An adjunction theorem for locally equiconnected spaces},
\newblock \red{Pacific J. Math.} {\bf 41} (1972), 669--685.
\bibitem{Fadell}
E.~Fadell and S. Husseini,
\newblock {\em Relative category, products and coproducts}, 
\newblock \red{Seminario Matematico e Fisico di Milano} {\bf 64} (1994), 99--115.
\bibitem{Farber2}
Michael Farber,
\newblock {\em Instabilities of robot motion},
\newblock Topology and its Applications {\bf 140} (2004), no. 2-3, 245--266.
\bibitem{JoseManuel}
J.M.~Garc\'ia-Calcines,
\emph{A note on covers defining relative and sectional categories},
Topology and its Applications {\bf 265} (2019), 106810.
\bibitem{CalcinesVan}
J.M.~Garc\'ia-Calcines and L.~Vandembroucq,
\emph{Weak sectional category},
\red{Journal of the London Mathematical Society}  {\bf 82} (2010), no. 3, 621--642.
\bibitem{GGY2016polyhedral}
J. Gonz\'alez, B. Guti\'errez and S. Yuzvinsky,
\emph{Higher topological complexity of subcomplexes of products of spheres and related polyhedral product spaces},
\red{Topological Methods in Nonlinear Analalysis} {\bf 48} (2016), no. 2, 419--451.
\bibitem{Iwase}
Norio Iwase and Michihiro Sakai,
\newblock {\em Topological complexity is a fibrewise L-S category},
\newblock Topology and its Applications {\bf 157} (2010), no. 1, 10--21.
\bibitem{Iwase2}
Norio Iwase and Michihiro Sakai, 
\newblock {\em Erratum to ``Topological complexity is a fibrewise L-S category" [Topology Appl. 157 (1) (2010) 10-21]},
\newblock Topology and its Applications. {\bf 159} (2012), no. 10-11, 2810--2813.
\bibitem{Lundell}
Albert T. Lundell and Stephen Weingram,
\newblock {\em The topology of $\CW$ complexes}, 
\newblock Van Nostrand Reinhold Company, New York, 1969.
\bibitem{Lupton}
Gregory Lupton and J\'er\^{o}me Scherer,
\newblock {\em Topological complexity of H-spaces},
\newblock \red{Proceedings of the American Mathematical Society} {\bf 141} (2013), no. 5, 1827--1838.
\end{thebibliography}
\end{document}